\documentclass[11pt]{amsart}
\usepackage{amsbsy}
\usepackage{amsmath}
\usepackage[alphabetic, backrefs, msc-links]{amsrefs}
\usepackage{amssymb}
\usepackage{amsthm}
\usepackage{amsxtra}
\usepackage{cancel}
\usepackage{comment}
\usepackage{fullpage}
\usepackage{graphicx}
\usepackage{hyperref}
\usepackage{paralist}
\usepackage{mathrsfs}
\usepackage{stmaryrd}
\usepackage{textcomp}
\usepackage{enumitem}
\usepackage{pxfonts}
\usepackage{fancyhdr}
\usepackage{tikz}
\usetikzlibrary{arrows,decorations.pathmorphing}
\usepackage{bbm}
\makeatletter
\pgfqkeys{/tikz/cs}{
  latitude/.store in=\tikz@cs@latitude,
  longitude/.style={angle={#1}},
  theta/.style={latitude={#1}},
  rho/.style={angle={#1}}
}
\tikzdeclarecoordinatesystem{xyz spherical}{
  \pgfqkeys{/tikz/cs}{angle=0,radius=0,latitude=0,#1}%
  \pgfpointspherical{\tikz@cs@angle}{\tikz@cs@latitude}{\tikz@cs@xradius}
}
\makeatother

\tikzset{my color/.code=\pgfmathparse{(#1+90)/180*100}\pgfkeysalso{every path/.style={color=red!\pgfmathresult!blue}}}

\DeclareRobustCommand{\SkipTocEntry}[5]{}

\newtheorem{prop}{Proposition}[subsection]

\newtheorem{thm}[prop]{Theorem}
\newtheorem{cor}{Corollary}[prop]
\newtheorem{lem}[cor]{Lemma}

\theoremstyle{definition}
\newtheorem{defn}[prop]{Definition}

\newtheorem{notn}[prop]{Notation}

\theoremstyle{remark}

\newtheorem{rem}[cor]{Remark}
\numberwithin{equation}{prop}

\DeclareMathOperator{\ad}{ad}

\DeclareMathOperator{\Arr}{Arr}

\DeclareMathOperator{\Aut}{Aut}

\DeclareMathOperator{\BAut}{\mathbf{Aut}}

\DeclareMathOperator{\Bun}{Bun}

\DeclareMathOperator{\Cart}{Cart}
\DeclareMathOperator{\Cat}{Cat}

\DeclareMathOperator{\cod}{cod}

\DeclareMathOperator{\End}{End}

\DeclareMathOperator{\Ext}{Ext}

\DeclareMathOperator{\Grp}{Grp}
\DeclareMathOperator{\Grpd}{Grpd}

\DeclareMathOperator{\Hom}{Hom}

\DeclareMathOperator{\id}{id}

\DeclareMathOperator{\Lie}{Lie}

\DeclareMathOperator{\Map}{Map}

\DeclareMathOperator{\Obj}{Obj}

\DeclareMathOperator{\res}{res}

\DeclareMathOperator{\Spec}{Spec}

\newcommand{\natmap}[5]{\xymatrix{
#1 \ar@/^/[r]^{#3} \ar@/_/[r]_{#4} \ar@{}[r]|{\Downarrow_{#5}} & #2
}}
\newcommand{\adjoints}[4]{\xymatrix{#1 \ar@/^/[rr]^{#3} & \perp & #2 \ar@/^/[ll]^{#4}}}
\newcommand{\smalladjoints}[4]{\xymatrix@C=5pt{#1 \ar@/^/[rr]^{#3} & \perp & #2 \ar@/^/[ll]^{#4}}}
\newcommand{\xtworightarrows}[4]{\xymatrix{#1 \ar@<.5ex>[r]^{#3} \ar@<-.5ex>[r]_{#4} & #2}}
\newcommand{\xtwoleftarrows}[4]{\xymatrix{#1 &  \ar@<-.5ex>[l]_{#3} \ar@<.5ex>[l]^{#4}  #2}}
\newcommand{\xrightleftarrows}[4]{\xymatrix{#1 \ar@/^/[r]^{#3} & #2 \ar@/^/[l]^{#4}}}

\newcommand{\op}[0]{\mathrm{op}}

\newcommand{\pr}[0]{\mathrm{pr}}

\title{Extensions of $\infty$-group sheaves}
\author{P\'al Zs\'amboki}
\email{pzsambok@uwo.ca}
\address{Department of Mathematics \\
Middlesex College \\
The University of Western Ontario \\
London, Ontario \\
Canada, N6A 5B7}

\begin{document}

\begin{abstract}

Let $\mathscr X$ be an $\infty$-topos, for example the $\infty$-category of simplicial sheaves on a Grothendieck site. Then $\infty$-group sheaves are group objects in $\mathscr X$. Let $A\in\Grp\mathscr X$ be such a group object. Then as $\mathscr X$ is an $\infty$-topos, there exists a universal $\mathbf BA$-fiber bundle $\mathbf BA\sslash\BAut A\xrightarrow q\mathbf B\BAut A$. We make $q$ pointed, and show that as a pointed map, via the looping-delooping equivalence, it is a universal extension of group objects by $A$. In particular, semidirect products of group objects by $A$ are classified by $\mathbf BA\sslash\BAut A$.

\end{abstract}

\maketitle

\tableofcontents

\section*{Introduction}

The construction of Lie algebras on group sheaves as described in \cite{sga3}*{Expos\'e II} naturally involves extensions of group sheaves.  We are interested in generalizing this construction to $\infty$-group sheaves and so were motivated to first examine extensions of $\infty$-group sheaves. Let us first review the Lie algebra of group sheaves construction.

Let $\mathscr X$ be a big site of a scheme $S$, and let $G\in\Grp\mathscr X$ be a group sheaf. Then the tangent sheaf $TG=\Hom_S(\Spec(\mathscr O_S\oplus\varepsilon\mathscr O_S),G)$ fits into a semi-direct product of group sheaves
$$
0\to\Lie G\to TG\to G\to1.
$$
This corresponds to the adjoint action map $G\xrightarrow\ad\Aut_{\Grp}(\Lie G)$. One can show that this action is $\mathscr O_S$-linear. Applying the $\Lie$ functor to this action map, we get an $\mathscr O_S$-linear morphism
$$
\Lie G\to\Lie(\Aut_{\mathscr O_S}\Lie G)\cong\End_{\mathscr O_S}(\Lie G),
$$
which in turn gives the Lie bracket map.

One can see that in order to generalize this construction to the case of $\infty$-group sheaves, we need to get a handle on semidirect products of $\infty$-group sheaves. In particular, we need to have an analogue of the adjoint action map classifying a semidirect product. In this article, we classify semidirect products of $\infty$-group sheaves.

In section 1.1, we review small object classifiers in $\infty$-topoi \cite{lurie2009higher}*{\S6.1.6}. In section 1.2, we review how one can use small object classifiers to classify fiber bundles \cite{nikolaus2014principal}*{\S4.1}. Let $A\in\Grp\mathscr X$ be a group object. Then there exists a universal $\mathbf BA$-fiber bundle $\mathbf BA\sslash\BAut A\xrightarrow q\mathbf B\BAut A$ in $\mathscr X$. After establishing some preliminary results about pointed $\infty$-categories in section 2.1, we prove our classification statement in section 2.2. We make $q$ pointed, and show that as a pointed map, via the looping-delooping equivalence, it is a universal extension of group objects by $A$. In particular, semidirect products of group objects are classified by $\mathbf BA\sslash\BAut A$.

Our interest in Lie algebras of $\infty$-group sheaves comes from the following. Let $G$ be an algebraic group on a scheme $S$, and let $X\xrightarrow fS$ be a morphism of schemes. Via the adjoint action map $G\xrightarrow\ad\Aut(\Lie G)$, we can get a twisting map $f_*\mathbf BG\to f_*\mathrm{Form}(\Lie G)$, where the codomain is the stack of families of Lie algebras locally isomorphic to $\Lie G$. It might be possible to compactify $f_*\mathrm{Form}(\Lie G)$ by letting the Lie algebras degenerate to Lie algebras on perfect complexes, which in turn could give information about limits in the stack of families of $G$-torsors $f_*\mathbf BG$. But for this, we need to handle $\Lie G$, in the case where $G$ is the automorphism $\infty$-group of forms of a perfect complex.

An analogous result to ours has been obtained in \cite{jardine2015local}*{Theorem 9.66}, but there only extensions of groupoid sheaves by group sheaves are classified, and not $\infty$-groups, which we need so that we can encode higher homotopies of perfect complexes.

We would like to thank Ajneet Dhillon, Nicole Lemire and Chris Kapulkin for the fruitful conversations on the topic of the article.

\section{Groups and principal bundles in an $\infty$-topos}

\subsection{Small object classifiers in an $\infty$-topos}

\begin{notn}

Let $\mathscr X$ be an $\infty$-category. Then we let $\Arr {\mathscr X}=\mathscr X^{\Delta^1}\xrightarrow{\res_1}\mathscr X$ denote the \emph{codomain fibration}. Let $S$ be a collection of morphisms in $\mathscr X$. Then let $\Arr^S{\mathscr X}\subseteq\Arr{\mathscr X}$ denote the full subcategory spanned by $S$. Suppose that $\mathscr X$ has pullbacks, and that $S$ is closed under pullbacks. Then we denote by $\Cart{\mathscr X}\subseteq\Arr_{\mathscr X}$ the 2-full subcategory with morphisms the pullback squares, and by $\Cart^S{\mathscr X}=\Cart{\mathscr X}\cap\Arr^S{\mathscr X}$.

\end{notn}

\begin{rem}

The target map $\Arr{\mathscr X}\to\mathscr X$ is a coCartesian fibration classified by the map $\mathscr X\xrightarrow[f\mapsto f_!]{X\mapsto\mathscr X_{/X}}\Cat_\infty$. Suppose that $\mathscr X$ admits pullbacks. Then the target map is also a Cartesian fibration, classified by $\mathscr X^\op\xrightarrow[f\mapsto f^*]{X\mapsto\mathscr X_{/X}}\Cat_\infty$. Its restriction to $\Cart{\mathscr X}$ is a right fibration classified by $\mathscr X^\op\xrightarrow{X\mapsto(\mathscr X_{/X})^\simeq}\mathscr S$.

\end{rem}

\begin{defn}

Let $\mathscr X$ be an $\infty$-category which admits pullbacks, and let $S$ be a collection of morphisms of $\mathscr X$ which is closed under pullbacks. Then we say that a morphism $X\xrightarrow\pi Y$ in $\mathscr X$ \emph{classifies $S$}, if it is a final object of the category $\Cart^S{\mathscr X}$. In this case, we also say that \emph{$Y\in\mathscr X$ is a classifying object for $S$}, and that $\pi$ is a \emph{universal map of property $S$}.

\end{defn}

\begin{defn}

Suppose that $\pi$ classifies $S$. Consider the following zigzag.
$$
\Cart^S{\mathscr X}\leftarrow(\Cart^S{\mathscr X})_{/\pi}\xrightarrow{\cod_{/\pi}}\mathscr X_{/Y}
$$
The left-hand arrow is a trivial fibration by definition. Since the map $\Cart^S\mathscr X\xrightarrow\cod\mathscr X$ is the restriction to the full subcategory of $\cod$-Cartesian edges of the Cartesian fibration $\Arr^S(\mathscr X)\xrightarrow\cod\mathscr X$, the right-hand arrow is a trivial fibration by Lemma \ref{lem:overcat of right fib}. This explains why $\pi$ is called a universal map of property $S$: every morphism $X'\xrightarrow fY'$ of property $S$ fits into an essentially unique Cartesian diagram of the form
\begin{center}

\begin{tikzpicture}[scale=1.5]

\node (Pnb) at (0,1) {$X'$};
\node (Pn) at (1,1) {$X$};
\node (Gnb) at (0,0) {$Y'$};
\node (Gn) at (1,0) {$Y.$};
\node at (0.5,0.5) {$\lrcorner$};
\path[->,font=\scriptsize,>=angle 90]
(Pnb) edge (Pn)
(Pnb) edge node [right] {$f$} (Gnb)
(Gnb) edge node [above] {$\mathbf c_f$} (Gn)
(Pn) edge node [right] {$\pi$} (Gn);

\end{tikzpicture}

\end{center}
In this situation, the map $\mathbf c_{f}$ is called the map \emph{classifying $f$}.

\end{defn}

\begin{lem}\label{lem:overcat of right fib}

Let $K\xrightarrow pK'$ be a right fibration of simplicial sets, $x\in K$ a vertex, and $K_{/x}\xrightarrow{p_{/x}}K'_{/p(x)}$ the induced map on overcategories. Then the map $p_{/x}$ is a trivial fibration.

\end{lem}

\begin{proof}

Let $n\ge0$. Then a lifting problem
\begin{center}

\begin{tikzpicture}[scale=1.5]

\node (Pnb) at (0,1) {$\partial\Delta^n$};
\node (Pn) at (1,1) {$K_{/x}$};
\node (Gnb) at (0,0) {$\Delta^n$};
\node (Gn) at (1,0) {$K'_{/p(x)}$};
\path[->,font=\scriptsize,>=angle 90]
(Pnb) edge (Pn)
(Pnb) edge (Gnb)
(Gnb) edge (Gn)
(Pn) edge node [right] {$p_{/x}$} (Gn);

\end{tikzpicture}

\end{center}
corresponds to a lifting problem
\begin{center}

\begin{tikzpicture}[scale=1.5]

\node (Pnb) at (0,1) {$\Lambda^{n+1}_{n+1}$};
\node (Pn) at (1,1) {$K$};
\node (Gnb) at (0,0) {$\Delta^{n+1}$};
\node (Gn) at (1,0) {$K',$};
\path[->,font=\scriptsize,>=angle 90]
(Pnb) edge (Pn)
(Pnb) edge (Gnb)
(Gnb) edge (Gn)
(Pn) edge node [right] {$p$}(Gn);

\end{tikzpicture}

\end{center}
which has a solution, since $p$ is a right fibration.
\end{proof}

\begin{thm}\cite{lurie2009higher}*{Theorem 6.1.6.8}

Let $\mathscr X$ be a presentable $\infty$-category. Then $\mathscr X$ is an $\infty$-topos if and only if the following conditions are satisfied.

\begin{enumerate}

\item Colimits in $\mathscr X$ are universal.

\item For all sufficiently large cardinals $\kappa$, there exists a universal relatively $\kappa$-compact morphism in $\mathscr X$.

\end{enumerate}

\end{thm}

\begin{notn}

Let $\kappa$ be a regular cardinal, and let $S$ denote the collection of relatively $\kappa$-compact morphisms. Then we let $\Cart^\kappa{\mathscr X}=\Cart^S{\mathscr X}$. We denote a universal relatively $\kappa$-compact morphism by $\widehat\Obj_\kappa\xrightarrow\pi\Obj_\kappa$.

\end{notn}

\subsection{Principal bundles and fiber bundles}

In this subsection, let us fix an $\infty$-topos $\mathscr X$.

\begin{defn}

Let $G\in\Grp\mathscr X$ be a group object and $P\in\mathscr X$ an object. Then a \emph{(right) $G$-action on $P$} is a morphism $P_\bullet\to G$ in $\Grpd\mathscr X$ such that for all morphisms $[m]\to[n]$ in $\Delta$, the square
\begin{center}

\begin{tikzpicture}[scale=1.5]

\node (Pnb) at (0,1) {$P_n$};
\node (Pn) at (1,1) {$P_m$};
\node (Gnb) at (0,0) {$G_n$};
\node (Gn) at (1,0) {$G_m$};
\path[->,font=\scriptsize,>=angle 90]
(Pnb) edge (Pn)
(Pnb) edge (Gnb)
(Gnb) edge (Gn)
(Pn) edge (Gn);

\end{tikzpicture}

\end{center}
is Cartesian, and we have $P_0\simeq P$. Let $\mathop{\text{Action-$G$}}\subseteq\Grpd(\mathscr X)_{/G}$ denote the full subcategory of right $G$-actions.

\end{defn}

\begin{rem}

Let $G$ be a discrete group acting on a set $P$ from the right via the action map $P\times G\xrightarrow\rho P$. Then the square
\begin{center}

\begin{tikzpicture}[scale=1.5]

\node (Pnb) at (0,1) {$P\times G$};
\node (Pn) at (1,1) {$P$};
\node (Gnb) at (0,0) {$G$};
\node (Gn) at (1,0) {$\ast$};
\path[->,font=\scriptsize,>=angle 90]
(Pnb) edge node [above] {$\rho$} (Pn)
(Pnb) edge (Gnb)
(Gnb) edge (Gn)
(Pn) edge (Gn);

\end{tikzpicture}

\end{center}
is Cartesian.

\end{rem}

\begin{rem}

\cite{nikolaus2014principal}*{Proposition 3.15} shows that it is enough to assume less, and the proof needs even less than assumed.

\end{rem}

\begin{thm}\cite{nikolaus2014principal}*{Theorem 3.19}

Let $G\in\Grp\mathscr X$ be a group object, and $X\in\mathscr X$ be an object. Then principal bundles are classified by morphisms $X\to\mathbf BG$. That is, we have an equivalence
$$
(\mathop{\text{Action-$G$}})\times_\mathscr X\{X\}\simeq\Map(X,\mathbf BG)
$$
where the map $\mathop{\text{Action-$G$}}\to\mathscr X$ is given by taking a map of groupoids $P\to G$ to $\varinjlim P$. The equivalence is given by taking geometrical realizations and restricting to $-1\in\Delta_+$.

\end{thm}

\begin{defn}

Let $V,X\in\mathscr X$ be objects. Then a \emph{$V$-fiber bundle over $X$} is a map $E\xrightarrow pX$ such that there exists an effective epimorphism $U\twoheadrightarrow X$, and a Cartesian square
\begin{center}

\begin{tikzpicture}[scale=1.5]

\node (A) at (0,1) {$U\times V$};
\node (B) at (0,0) {$U$};
\node (X) at (1,1) {$E$};
\node (Y) at (1,0) {$X$.};
\node (c) at (0.7,0.3) {$\lrcorner$};
\path[->,font=\scriptsize,>=angle 90]
(A) edge (X)
(A) edge node [right] {$\pr_U$} (B)
(X) edge node [right] {$p$} (Y)
(B) edge [->>] (Y);

\end{tikzpicture}

\end{center}
The \emph{space of $V$-fiber bundles} is the full subcategory $\Bun_V\subseteq\Cart_{\mathscr X}$ of $V$-fiber bundles.

\end{defn}

\begin{prop}

Let $V\in\mathscr X$ be a small object, ie. a $\kappa$-compact object for some regular cardinal $\kappa$. Then every $V$-fiber bundle is relatively $\kappa$-compact.

\end{prop}

\begin{proof}

This is a special case of \cite{lurie2009higher}*{Lemma 6.1.6.5}.

\end{proof}

\begin{cor}

We have $\Bun_V\subseteq\Cart_{\mathscr X}^\kappa$.

\end{cor}

\begin{defn}\label{defn:BAut V}

Let $V\in\mathscr X$ be a $\kappa$-compact object. Then its \emph{inner automorphism group} $\BAut V$ with an action $\rho_V:\BAut V\mathop{\text{\rotatebox[origin=c]{-90}{$\circlearrowleft$}}}V$ is defined via the following diagram of Cartesian squares.
\begin{center}

\begin{tikzpicture}[xscale=2,yscale=1.5]

\node (V) at (0,1) {$V$};
\node (V/) at (1,1) {$V\sslash\BAut V$};
\node (HObj) at (2,1) {$\widehat\Obj_\kappa$};
\node (*) at (0,0) {$\ast$};
\node (BAut) at (1,0) {$\mathbf B\BAut V$};
\node (Obj) at (2,0) {$\Obj_\kappa$};
\node at (0.7,0.3) {$\lrcorner$};
\node at (1.7,0.3) {$\lrcorner$};
\path[->,font=\scriptsize,>=angle 90]
(V) edge (V/)
(V) edge (*)
(V/) edge (HObj)
(V/) edge node [right] {$q$} (BAut)
(HObj) edge node [right] {$\pi$} (Obj)
(*) edge [->>] (BAut)
(BAut) edge [right hook->] (Obj)
(*) edge [bend right] node [below] {$\mathbf c_V$} (Obj);

\end{tikzpicture}

\end{center}

\end{defn}

\begin{prop}\cite{nikolaus2014principal}*{Proposition 4.10}

Let $E\xrightarrow pX$ be a $V$-fiber bundle. Then its classifying map $X\xrightarrow{\mathbf c_p}\Obj_\kappa$ factors through the monomorphism $\mathbf B\BAut V\hookrightarrow\Obj_\kappa$, resulting in the following pasting diagram.
\begin{center}

\begin{tikzpicture}[xscale=2,yscale=1.5]

\node (V) at (0,1) {$E$};
\node (V/) at (1,1) {$V\sslash\BAut V$};
\node (HObj) at (2,1) {$\widehat\Obj_\kappa$};
\node (*) at (0,0) {$X$};
\node (BAut) at (1,0) {$\mathbf B\BAut V$};
\node (Obj) at (2,0) {$\Obj_\kappa$};
\node at (0.7,0.3) {$\lrcorner$};
\node at (1.7,0.3) {$\lrcorner$};
\path[->,font=\scriptsize,>=angle 90]
(V) edge (V/)
(V) edge node [left] {$p$} (*)
(V/) edge (HObj)
(V/) edge node [right] {$q$} (BAut)
(HObj) edge node [right] {$\pi$} (Obj)
(*) edge (BAut)
(BAut) edge [right hook->] (Obj)
(*) edge [bend right] node [below] {$\mathbf c_p$} (Obj);

\end{tikzpicture}

\end{center}

\end{prop}

\begin{thm}\label{thm:classifying fiber bundles}

We have a zigzag of trivial fibrations
$$
\Bun_V\leftarrow(\Bun_V)_{/q}\to\mathscr X_{/\mathbf B\BAut V},
$$
that is $q$ is a universal $V$-fiber bundle.

\end{thm}

\begin{proof}

Let us denote the right-hand Cartesian square of Definition \ref{defn:BAut V} by $\alpha$. Then the map $\mathbf B\BAut V\xrightarrow i\Obj_\kappa$ is a monomorphism if and only if $i\in\mathscr X_{/\Obj_\kappa}$ is $(-1)$-truncated, which in turn is equivalent to $\alpha\in(\Cart_{\mathscr X}^\kappa)_{/\pi}$ being $(-1)$-truncated, that is $\alpha$ is a monomorphism. Therefore, the induced map $(\Cart_{\mathscr X}^\kappa)_{/q}\to(\Cart_{\mathscr X}^\kappa)_{/\pi}$ is fully faithful by Lemma \ref{lem:monomorphism on overcategories}. Thus, the commutative square
\begin{center}

\begin{tikzpicture}[xscale=3,yscale=2]
\node (Bq) at (0,1) {$(\Bun_V)_{/q}$};
\node (Cq) at (1,1) {$(\Cart_{\mathscr X}^\kappa)_{/q}$};
\node (Cpi) at (2,1) {$(\Cart_{\mathscr X}^\kappa)_{/\pi}$};
\node (B) at (0,0) {$\Bun_V$};
\node (C) at (2,0) {$\Cart_{\mathscr X}^\kappa$};
\node at (1.7,0.3) {$\lrcorner$};
\path[->,font=\scriptsize,>=angle 90]
(Bq) edge node [above] {$\cong$} (Cq)
(Bq) edge (B)
(Cq) edge [right hook->] node [above] {$\alpha\circ$} (Cpi)
(Cpi) edge node [right] {$\simeq$} (C)
(B) [right hook->] edge (C);
\end{tikzpicture}

\end{center}
is Cartesian, which proves our claim.

\end{proof}

\begin{lem}\label{lem:monomorphism on overcategories}

Let $X\xrightarrow fY$ be a morphism in an $\infty$-category $\mathscr C$. Suppose that $f$ is a monomorphism. Then the postcomposition map $\mathscr C_{/X}\xrightarrow{f\circ}\mathscr C_{/Y}$ is a monomorphism, ie.~fully faithful.

\end{lem}

\begin{proof}

It will be enough to show that the unit map $\id_{\mathscr C_{/X}}\xrightarrow\eta f^*f_!$ is an equivalence. Consider a bundle $P\xrightarrow pX$. It gives the following diagram.
\begin{center}

\begin{tikzpicture}[scale=1.5]
\node (P) at (0,1) {$P$};
\node (*!) at (1,1) {$f^*f_!P$};
\node (!) at (2,1) {$f_!P$};
\node (X) at (0,0) {$X$};
\node (Y) at (2,0) {$Y$};
\node at (1.7,0.3) {$\lrcorner$};
\path[->,font=\scriptsize,>=angle 90]
(P) edge node [above] {$\eta_P$} (*!)
(P) edge node [left] {$p$} (X)
(*!) edge node [above] {$(fp)^*f$} (!)
(*!) edge (X)
(!) edge node [right] {$f_!p$} (Y)
(X) edge node [above] {$f$} (Y);
\end{tikzpicture}

\end{center}
Since $f$ is a monomorphism, so is its pullback $(fp)^*f$. Since $((fp)^*f)\eta_P=\id_P$, by the uniqueness of the epi-mono factorization, the map $\eta_P$ is an equivalence, as required.

\end{proof}

\section{Classifying group extensions and semidirect products}

\begin{defn}

Let $\mathscr X$ be an $\infty$-topos, and $G,A$ two group objects in it. A \emph{group extension} of $G$ by $A$ is a fibration sequence $A\to\Hat G\to G$ in $\Grp\mathscr X$

\end{defn}

That is, group extensions of $G$ by $A$ are equivalent to pointed fibration sequences of the form $\mathbf BA\to\mathbf B\Hat G\to\mathbf BG$. If we could interpret the latter as $\mathbf BA$-fiber bundles in $\mathscr X_*$, then they could be classified as such. Unfortunately, this is not possible, due to the fact that if $\mathscr X$ is nontrivial, then $\mathscr X_*$ is not an $\infty$-topos.

\begin{prop}

Let $\mathscr X$ be a pointed $\infty$-topos. Then $\mathscr X$ is contractible.

\end{prop}

\begin{proof}

Let $X\in\mathscr X$ be an object. Then we have a canonical map $X\to\ast$. Since $\ast$ is also an initial object, $X$ needs to be initial too \cite{lurie2009higher}*{Lemma 6.1.3.6}. This shows that every object in $\mathscr X$ is initial, and thus $\mathscr X$ is contractible

\end{proof}

In Section 2.1, we show that the forgetful functor $\mathscr X_*\xrightarrow U\mathscr X$ has as left adjoint the adjoining a disjoint point functor $\mathscr X\xrightarrow{\sqcup\ast}\mathscr X_*$, and that a pointed square $(\Lambda^2_2)^{\vartriangleleft}\xrightarrow{\Bar\sigma}\mathscr X_*$ is Cartesian if and only if its image $(\Lambda^2_2)^{\vartriangleleft}\xrightarrow{U\Bar\sigma}\mathscr X$ is so. Therefore, we can define the $\infty$-category of group extensions by $A$ as follows.

\begin{defn}

The \emph{$\infty$-category of group extensions by $A$} is the full subcategory $\Ext_A\subseteq\Cart_{\mathscr X_*}$ on maps with codomain of the form $\mathbf BG$ for some $G\in\Grp\mathscr X$, and with fiber $\mathbf BA$. Note that by construction the domain is also the delooping of a group object. Let $G\in\Grp\mathscr X$ be a group object. Then $\Ext(G,A)=\Ext_A(\mathbf BG)$ is the \emph{classifying space of extensions of $G$ by $A$}.

\end{defn}

We can make the universal $\mathbf BA$-fiber bundle $\mathbf BA\sslash\BAut\mathbf BA\xrightarrow q\mathbf B\BAut\mathbf BA$ pointed in such a way that the Cartesian diagram
\begin{center}

\begin{tikzpicture}[xscale=2,yscale=1.5]
\node (BA) at (0,1) {$\mathbf BA$};
\node (BA/) at (1,1) {$\mathbf BA\sslash\BAut(\mathbf BA)$};
\node (*) at (0,0) {$*$};
\node (BAut) at (1,0) {$\mathbf B\BAut(\mathbf BA)$};
\path[->,font=\scriptsize,>=angle 90]
(BA) edge (BA/)
(BA) edge (*)
(BA/) edge node [right] {$q$} (BAut)
(*) edge node [above] {$\mathbf c_{\mathbf BA}$} (BAut);
\end{tikzpicture}

\end{center}
is pointed. Then in section 2.2, we show that $q$ as a pointed map is a universal group extension by $A$.

\subsection{Preliminaries about pointed $\infty$-categories}

\begin{notn}

Let $\mathscr X$ be an $\infty$-category with a final object. Then let $\mathscr X_*\xrightarrow U\mathscr X$ denote the forgetful map, and let $\mathscr X\xrightarrow F\mathscr X_*$ denote the map adding a disjoint point.

\end{notn}

\begin{prop}\label{prop:adjunction for pointed}

Let $\mathscr X$ be an $\infty$-category with a final object. Then the pair $(F,U)$ is an adjunction.

\end{prop}

\begin{proof}

Let $X\in\mathscr X$ and $(Y,y)\in\mathscr X_*$. We want the inclusion map $X\xrightarrow{\eta_X}X\sqcup\ast$ to serve as a unit map. Then we need to show that the composite
$$
\Map_{\mathscr X_*}(X\sqcup\ast,Y)\xrightarrow U\Map_\mathscr X(X\sqcup\ast,Y)\xrightarrow{\circ\eta_X}\Map_\mathscr X(X,Y)
$$
is an equivalence. By \cite{lurie2009higher}*{Lemma 5.5.5.12 and Theorem 4.2.4.1}, we have a pasting diagram
\begin{center}

\begin{tikzpicture}[xscale=4,yscale=2]
\node (*X*) at (0,1) {$\Map_{\mathscr X_*}(X\sqcup\ast,Y)$};
\node (X*) at (1,1) {$\Map_{\mathscr X}(X\sqcup\ast,Y)$};
\node (X) at (2,1) {$\Map_{\mathscr X}(X,Y)$};
\node (*0) at (0,0) {$\ast$};
\node (*Y) at (1,0) {$\Map_{\mathscr X}(\ast,Y)$};
\node (*2) at (2,0) {$\ast,$};
\node at (0.5,0.5) {$\lrcorner$};
\node at (1.5,0.5) {$\lrcorner$};
\path[->,font=\scriptsize,>=angle 90]
(*X*) edge node [above] {$U$} (X*)
(*X*) edge (*0)
(X*) edge node [above] {$\circ\eta_X$} (X)
(X) edge (*2)
(X*) edge node [right] {$\circ i_*$} (*Y)
(*0) edge node [above] {$\{y\}$} (*Y)
(*Y) edge (*2);
\end{tikzpicture}

\end{center}
which proves our claim.

\end{proof}

\begin{prop}\label{prop:Cartesian in pointed}

Let $\ast\star\Lambda_2^2\xrightarrow{\Bar\sigma_*}\mathscr X_*$ be a pointed square. Then it is Cartesian precisely when its image $\ast\star\Lambda_2^2\xrightarrow{U\Bar\sigma_*}\mathscr X$ is Cartesian.

\end{prop}

\begin{proof}

If $\Bar\sigma_*$ is Cartesian, then so is $U\Bar\sigma_*$, since $U$ is a right adjoint, which shows necessity. To prove sufficiency, let us suppose that $U\Bar\sigma_*$ is Cartesian. Let us denote by $\sigma_*$ the restriction $\Bar\sigma_*|\Lambda_2^2$. We need to show that for any $n\ge0$, any lifting problem of the form
\begin{center}

\begin{tikzpicture}[scale=1.5]

\node (delta) at (0,1) {$\partial\Delta^n$};
\node (square) at (1,1) {$(\mathscr X_*)_{/\Bar\sigma_*}$};
\node (Delta) at (0,0) {$\Delta^n$};
\node (horn) at (1,0) {$(\mathscr X_*)_{/\sigma_*}$};
\path[->,font=\scriptsize,>=angle 90]
(delta) edge (square)
(delta) edge (Delta)
(square) edge (horn)
(Delta) edge (horn);

\end{tikzpicture}

\end{center}
has a solution. But such a lifting problem is the same as a lifting problem of the form
\begin{center}

\begin{tikzpicture}[xscale=2,yscale=1.5]

\node (delta) at (0,1) {$\ast\star\partial\Delta^n$};
\node (square) at (1,1) {$(\mathscr X)_{/U\Bar\sigma_*}$};
\node (Delta) at (0,0) {$\ast\star\Delta^n$};
\node (horn) at (1,0) {$(\mathscr X)_{/U\sigma_*}$};
\path[->,font=\scriptsize,>=angle 90]
(delta) edge (square)
(delta) edge (Delta)
(square) edge (horn)
(Delta) edge (horn);

\end{tikzpicture}

\end{center}
which has a solution, since $U\Bar\sigma_*$ is Cartesian.

\end{proof}

\subsection{Classifying group extensions and semidirect products}

Now we are ready to classify group extensions. Let us fix a group object $A\in\Grp\mathscr X$.

\begin{thm}\label{thm:classifying group extensions}

Let $\mathbf BA\sslash\BAut(\mathbf BA)\xrightarrow q\mathbf B\BAut(\mathbf BA)$ denote the universal $\mathbf BA$-fiber bundle in $\mathscr X$. Let us make $q$ a pointed map in such a way that the Cartesian diagram
\begin{center}

\begin{tikzpicture}[xscale=2,yscale=1.5]
\node (BA) at (0,1) {$\mathbf BA$};
\node (BA/) at (1,1) {$\mathbf BA\sslash\BAut(\mathbf BA)$};
\node (*) at (0,0) {$*$};
\node (BAut) at (1,0) {$\mathbf B\BAut(\mathbf BA)$};
\path[->,font=\scriptsize,>=angle 90]
(BA) edge (BA/)
(BA) edge (*)
(BA/) edge node [right] {$q$} (BAut)
(*) edge node [above] {$\mathbf c_{\mathbf BA}$} (BAut);
\end{tikzpicture}

\end{center}
is pointed. Then $q$ is a universal group extension by $A$.

\end{thm}

\begin{cor}\label{cor:adjoint action}

Semidirect products of $G$ and $A$ are classified by pointed morphisms $\mathbf BG\to\mathbf BA\sslash\BAut A$.

\end{cor}

\begin{rem}

Suppose that $A$ is an $E_2$-group. Then an $\infty$-group extension $A\to\Hat G\to G$ is called \emph{central}, when the map $\mathbf BA\to\mathbf B\Hat G$ is deloopable. The fibration sequence $\mathbf BA\to*\to\mathbf B^2A$ is classified by a map $\mathbf B^2A\xrightarrow Z\mathbf B\BAut\mathbf BA$. Therefore, a group extension of $G$ by $A$ is central precisely when its classifying map factors through $Z$. This is why the classifying space of central extensions is defined as $\Map(\mathbf BG,\mathbf B^2A)$ in \cite{nikolaus2014principal}. Note also that the map classifying a semidirect product which is central factors through $\ast$, and is thus trivial.

\end{rem}

\begin{proof}[Proof of Theorem \ref{thm:classifying group extensions}]

We claim that $q\in\Ext_A$ is a final object, that is the restriction map $(\Ext_A)_{/q}\to\Ext_A$ is a trivial fibration. Let $n\ge0$, and suppose given a lifting problem of the following form
\begin{center}

\begin{tikzpicture}[xscale=2,yscale=1.5]
\node (simplex) at (0,0) {$\Delta^n$};
\node (boundary) at (0,1) {$\partial\Delta^n$};
\node (extq) at (1,1) {$(\Ext_A)_{/q}$};
\node (ext) at (1,0) {$\Ext_A.$};
\path[->,font=\scriptsize,>=angle 90]
(boundary) edge (simplex)
(simplex) edge node [above] {$\tau$} (ext)
(boundary) edge node [above] {$\sigma$} (extq)
(extq) edge (ext);
\end{tikzpicture}

\end{center}
This gives rise to the following lifting problem
\begin{center}

\begin{tikzpicture}[xscale=4,yscale=1.5]
\node (simplex) at (0,0) {$\Delta^{[0,n+1]}\sqcup_{\Lambda^{[0,n+1]}_0}\Delta^{\widehat{[1,n+1]}}$};
\node (boundary) at (0,1) {$\Delta^{[1,n+1]}\sqcup_{\partial\Delta^{[1,n+1]}}\Lambda^{[1,n+2]}_{n+2}$};
\node (extq) at (1,1) {$\Cart_{\mathscr X_*}$};
\node (ext) at (1,0) {$\mathscr X_*,$};
\path[->,font=\scriptsize,>=angle 90]
(boundary) edge (simplex)
(simplex) edge node [above] {$(\sigma,\tau)'$} (ext)
(boundary) edge node [above] {$(U\sigma,U\tau)'$} (extq)
(extq) edge (ext);
\end{tikzpicture}

\end{center}
where we denote by $\Delta^{\widehat{[1,n+1]}}\subseteq\Delta^{n+2}$ the sub-simplicial set with all subsets of $[0,n+2]$ not containing $[1,n+1]$. One can check that the inclusion $\Delta^{[1,n+1]}\sqcup_{\partial\Delta^{[1,n+1]}}\Lambda^{[1,n+2]}_{n+2}\subset\Delta^{[0,n+1]}\sqcup_{\Lambda^{[0,n+1]}_0}\Delta^{\widehat{[1,n+1]}}$ is actually $\partial\Delta^{[1,n+2]}\subset\Lambda^{n+2}_0$. This inclusion is right anodyne, because a lifting problem along it can be understood as a lifting problem along $\partial\Delta^{[1,n+1]}\times\{1\}\subset\partial\Delta^{[1,n+1]}\times\Delta^1$ by extending $\Delta^0$ to $\partial\Delta^{[1,n+1]}\times\{0\}$ in a degenerate way, and that is right anodyne \cite{lurie2009higher}*{Proposition 2.1.2.6${}^\op$}. Let $\Lambda^{n+2}_0\xrightarrow{\tau_{\mathbf BA}}\Cart_{\mathscr X_*}$ be a lift. Then by Proposition \ref{prop:Cartesian in pointed}, the restriction $U\tau_{\mathbf BA}$ gives a lifting problem of the form
\begin{center}

\begin{tikzpicture}[xscale=2,yscale=1.5]
\node (simplex) at (0,0) {$\Delta^{n+1}$};
\node (boundary) at (0,1) {$\partial\Delta^{n+1}$};
\node (extq) at (1,1) {$(\Bun_{\mathbf BA})_{/q}$};
\node (ext) at (1,0) {$\Bun_{\mathbf BA}$.};
\path[->,font=\scriptsize,>=angle 90]
(boundary) edge (simplex)
(simplex) edge (ext)
(boundary) edge (extq)
(extq) edge (ext);
\end{tikzpicture}

\end{center}
Let us denote a solution of that by $\Delta^{n+1}\xrightarrow{(U\sigma)_{\mathbf BA}}(\Bun_{\mathbf BA})_{/q}$.

Let $\Lambda^{n+2}_0\xrightarrow{\tau_{\mathbf BA}}\Cart_{\mathscr X_*}$ define a diagram $\Lambda^{[0,\widehat{n+3},n+4]}_0\xrightarrow{\tau'}\Arr\mathscr X$. Let $\Delta^{n+1}\xrightarrow{(U\sigma)_{\mathbf BA}}(\Bun_{\mathbf BA})_{/q}$ define a diagram $\Delta^{[1,n+3]}\to\Arr\mathscr X$, which we can extend to $d_0\xrightarrow{\sigma'}\Arr\mathscr X$ along the degenerate edge $(n+3)\xrightarrow r(n+4)$. Then we get a diagram $\Lambda^{[0,\widehat{n+3},n+4]}_0\cup d_0\xrightarrow{(\tau',\sigma')'}\Arr\mathscr X$. Note that $\Lambda^{[0,\widehat{n+3},n+4]}_0\cup d_0=(\partial d_{n+3})\cup d_0$. Then we can get a lift $\Delta^{n+4}\xrightarrow{\Tilde\sigma}\Arr\mathscr X$ along the inner anodyne inclusions $(\partial d_{n+3})\cup d_0\subset\Lambda^{n+4}_{n+3}$ (Lemma \ref{lem:make it pointed}) and $\Lambda^{n+4}_{n+3}\subset\Delta^{n+4}$. The restriction $\Tilde\sigma d_{n+3}$ will be a solution to the original lifting problem.

\end{proof}

\begin{lem}\label{lem:make it pointed}

Let $n\ge0$. Then the inclusion $(\partial d_{n+3})\cup d_0\subset\Lambda^{n+4}_{n+3}$ is inner anodyne.

\end{lem}

\begin{proof}

It will be enough to show that the intersections with $d_j$ for $j\in\{0,\dotsc,\widehat{n+3},n+4\}$ are inner anodyne. For $j=0$, we have $d_0\subseteq d_0$. For $0<j<n+3$, we have $d_{0j}\cup d_{j(n+3)}\subset d_j$, which is of the form $\Delta^{[0,j-1]}\star\Delta^{[j+1,\widehat{n+3},n+4]}\cup\Delta^{[1,j-1]}\star\Delta^{[j+1,n+4]}\subset\Delta^{[0,j-1]}\star\Delta^{[j+1,n+4]}$, and that is inner anodyne, because $\Delta^{[1,j-1]}\subset\Delta^{[0,j-1]}$ is right anodyne by Lemma \ref{lem:subsimplex right anodyne}. For $j=n+4$, we have $d_{0(n+3)}\cup d_{(n+3)(n+4)}\subset d_{n+4}$, which is of the form $\Delta^{[0,1]}\star\Delta^{[2,n+2]}\cup\Delta^{\{1\}}\star\Delta^{[2,n+4]}\subset\Delta^{[0,1]}\star\Delta^{[2,n+4]}$, and that is inner anodyne, because $\Delta^{\{1\}}\subset\Delta^{[0,1]}$ is right anodyne by Lemma \ref{lem:subsimplex right anodyne}.

\end{proof}

\begin{lem}\label{lem:subsimplex right anodyne}

Let $\ell>0$. Then for all $0<k\le\ell$, the inclusions $\Delta^{[k,\ell]}\hookrightarrow\Delta^\ell$ are right anodyne.

\end{lem}

\begin{proof}

Let us prove this by induction on $\ell$. This follows by definition for $\ell=1$. Let $\ell>1$ and $0<k\le\ell$. Since the inclusion $\Delta^{[k,\ell]}\subset\Delta^{[1,\ell]}$ is right anodyne by hypothesis, we can assume that $k=1$. Then we can see that the inclusion $\Delta^{[1,\ell]}\subset\Lambda^\ell_\ell$ is right anodyne by hypothesis, thus we are done.

\end{proof}

\section*{Conclusion}

As stated in Corollary \ref{cor:adjoint action}, a semi-direct product of $\infty$-group sheaves $1\to A\to\Hat G\to G\to1$ is classified by a pointed map $\mathbf BG\to\mathbf BA\sslash\BAut\mathbf BA$, which we can view as the $\infty$-categorical version of the adjoint action $G\to\Aut A$.

Towards constructing Lie algebras of $\infty$-group sheaves, the next step will be to specialize to the case of the semidirect product $0\to\Lie G\to TG\to G\to1$, which can be constructed analogously to the classical way in Homotopical Algebraic Geometry \cite{toen2008homotopical}*{\S1.4.1}. Following \cite{sga3}*{Expos\'e II}, we will first have to show that the adjoint action $\mathbf BG\to\mathbf B(\Lie G)\sslash\BAut\mathbf B(\Lie G)$ is $\mathscr O_S$-linear.


\begin{bibdiv}
\begin{biblist}

\bib{sga3}{book}{
   title = {Sch\'emas en groupes (SGA 3)},
   author = {Artin, Michael},
   author = {Bertin, Jean--\'Etienne},
   author = {Demazure, Michel},
   author = {Gabriel, Pierre},
   author = {Grothendieck, Alexander},
   author = {Raynaud, Michel},
   author = {Serre, Jean--Pierre},
   editor = {Gille, Philippe},
   editor = {Polo, Patrick},
   publisher = {Soci\'et\'e Math\'ematique de France},
   year = {2011},
   label = {SGA3},
   }
   
   \bib{jardine2015local}{book}{
   author={Jardine, John F.},
   title={Local homotopy theory},
   series={Springer Monographs in Mathematics},
   publisher={Springer, New York},
   date={2015},
   pages={x+508},
   isbn={978-1-4939-2299-4},
   isbn={978-1-4939-2300-7},
   review={\MR{3309296}},
   doi={10.1007/978-1-4939-2300-7},
}

\bib{lurie2009higher}{book}{
   author={Lurie, Jacob},
   title={Higher topos theory},
   series={Annals of Mathematics Studies},
   volume={170},
   publisher={Princeton University Press, Princeton, NJ},
   date={2009},
   pages={xviii+925},
   isbn={978-0-691-14049-0},
   isbn={0-691-14049-9},
   review={\MR{2522659 (2010j:18001)}},
   doi={10.1515/9781400830558},
   eprint={http://www.math.harvard.edu/~lurie/papers/highertopoi.pdf},
}

\bib{nikolaus2014principal}{article}{
date={2014},
issn={2193-8407},
journal={Journal of Homotopy and Related Structures},
doi={10.1007/s40062-014-0083-6},
title={Principal $\infty$-bundles: general theory},
url={http://dx.doi.org/10.1007/s40062-014-0083-6},
publisher={Springer Berlin Heidelberg},
keywords={Nonabelian cohomology; Higher topos theory; Principal bundles},
author={Nikolaus, Thomas},
author={Schreiber, Urs},
author={Stevenson, Danny},
pages={1--53},
}

\bib{toen2008homotopical}{article}{
   author={To{\"e}n, Bertrand},
   author={Vezzosi, Gabriele},
   title={Homotopical algebraic geometry. II. Geometric stacks and
   applications},
   journal={Mem. Amer. Math. Soc.},
   volume={193},
   date={2008},
   number={902},
   pages={x+224},
   issn={0065-9266},
   review={\MR{2394633 (2009h:14004)}},
   doi={10.1090/memo/0902},
   eprint={http://perso.math.univ-toulouse.fr/btoen/files/2012/04/HAGII.pdf},
}

\end{biblist}
\end{bibdiv}
\end{document}